\documentclass[11pt]{amsart}
\usepackage{geometry}                
\usepackage{graphicx}
\usepackage{amssymb}
\usepackage{epstopdf}
\usepackage{biblatex}
\begin{document}

\title{Compact four-manifolds with pinched self-dual Weyl curvature}


\author{Inyoung Kim}
\maketitle
\begin{abstract}
We consider compact oriented four-manifolds with harmonic self-dual Weyl curvature 
 in addition to a pinching condition. 
\end{abstract}

\maketitle

\section{\large\textbf{Introduction}}\label{S:Intro}
Let $M$ be a compact K\"ahler-Einstein surface with nonpositive holomorphic bisectional curvature. 
Siu-Yang [18] showed $M$ is biholomorphically isometric to a compact quotient of the complex 2-ball with an invariant metric if
for some $a<\frac{2}{3(1+\sqrt{6/11})}$
\[H_{av}-H_{min}\leq a(H_{max}-H_{min}),\]
where $H_{max}(H_{min})$ denotes the maximum(minimum) of the holomorphic sectional curvature at each point respectively and $H_{av}$ the average
of the holomorphic sectional curvature at each point. 

Guan [11] showed that a compact K\"ahler-Einstein surface with nonpositive scalar curvature is a compact quotient
of the complex 2-dimensional unit ball or plane if 
\[H_{av}-H_{min}\leq \frac{1}{2}(H_{max}-H_{min}).\]

Let $(M, g)$ be an oriented riemannian four-manifold. 
The curvature operator $\mathfrak{R}:\Lambda^{2}\to\Lambda^{2}$ is given by 
 according to the decomposition of $\Lambda^{2}=\Lambda^{+}\oplus \Lambda^{-}$ as follows

\[ 
\mathfrak{R}=
\LARGE
\begin{pmatrix}

 \begin{array}{c|c}
\scriptscriptstyle{W^{+}\hspace{5pt}\scriptstyle{+}\hspace{5pt}\frac{s}{12}}I& \scriptscriptstyle{ric_{0}}\\
 \hline
 
 \hspace{10pt}
 
 \scriptscriptstyle{ric_{0}}& \scriptscriptstyle{W^{-}\hspace{5pt}\scriptstyle{+}\hspace{5pt}\frac{s}{12}}I\\
 \end{array}
 \end{pmatrix}
 \]
 where $s$ is the scalar curvature. 
 A 2-form $\omega\in\Lambda^{+}$ is self-dual if $*\omega=\omega$ and $\omega\in\Lambda^{-}$ is anti-self-dual  if $*\omega=-\omega$,
 where $*$ is the Hodge star operator defined by $\left<*\omega_{1}, \omega_{2}\right>dvol=\omega_{1}\wedge\omega_{2}$.
 When $ric_{0}=0$, $(M, g)$ is Einstein. 
 When $W^{\pm}=0$, $g$ is an anti-self-dual(self-dual) metric respectively.
 
\vspace{10pt}

\hrule

\vspace{10pt}
$\mathbf{Keywords}$: four-manifold, self-dual Weyl curvature, K\"ahler-Einstein

$\mathbf{MSC}$: 53C20, 53C21, 53C55

\vspace{10pt}

Republic of Korea, 

Email address: kiysd5@gmail.com

Let $(M, g)$ be an oriented riemannian four-manifold and let $\delta W^{+}:=-\nabla\cdot W^{+}$.
Then $(M, g)$ is said to have a harmonic self-dual Weyl curvature if $\delta W^{+}=0$.
Suppose $(M, g)$ is a compact oriented riemannian four-manifold with harmonic self-dual Weyl curvature
and nonpositive scalar curvature. 
Let $\lambda_{i}$ be the eigenvalues of $W^{+}$ such that $\lambda_{1}\leq\lambda_{2}\leq\lambda_{3}$. Suppose 
\[-8\left(1-\frac{\sqrt{3}}{2}\right)\lambda_{1}\leq \lambda_{3}\leq -2\lambda_{1}.\]
Polombo showed that $(M, g)$ is anti-self-dual [17]. We note that 
$-8\left(1-\frac{\sqrt{3}}{2}\right)\lambda_{1}\leq \lambda_{3}$ implies $\lambda_{1}+\lambda_{3}\geq 0$. 

Let $\lambda_{i}$ be the eigenvalues of $W^{-}$ such that $\lambda_{1}\leq \lambda_{2}\leq \lambda_{3}$. 
Then $\lambda_{1}+\lambda_{2}+\lambda_{3}=0$. 
We show that holomorphic sectional curvatures on a K\"ahler-Einstein surface are given by
\[H_{max}=\frac{s}{6}+\frac{\lambda_{3}}{2}, \hspace{10pt} H_{min}=\frac{s}{6}+\frac{\lambda_{1}}{2}. \]

Let $d\sigma$ be the volume form for the sphere of radius 1
and let $H$ be the holomorphic sectional curvature of a K\"ahler surface. 
The average of the holomorphic sectional curvature at each point is defined by 
\[H_{av}(p):=\frac{1}{vol(\mathbb{S}^{n-1})}\int_{S_{p}}Hd\sigma,\]
where $S_{p}$ is the unit sphere of $T_{p}M$. 
By Berger's Theorem [4], $H_{av}=\frac{s}{6}$. 
Thus, in case of K\"ahler-Einstein surfaces, the curvature condition
\[H_{av}-H_{min}\leq \frac{1}{2}(H_{max}-H_{min})\]
is equivalent to
\[\lambda_{1}+\lambda_{3}\geq 0.\]

Since $\lambda_{1}+\lambda_{2}+\lambda_{3}=0$, this is equivalent to $\lambda_{2}\leq 0$. 
Since $\lambda_{1}\leq 0$ and $\lambda_{3}\geq 0$, this curvature condition is equivalent to $det W^{-}=\lambda_{1}\lambda_{2}\lambda_{3}\geq 0$. 

Let $(M, g)$ be a compact oriented four-manifold with harmonic self-dual Weyl curvature.
It was shown that $(M, g)$ or its double cover is conformal to a K\"ahler metric with positive scalar curvature when $det W^{+}>0$ by LeBrun [15] and Wu [19].
In this paper, following Guan's method, we show that a compact riemannian four-manifold with harmonic self-dual Weyl curvature
and nonpositive scalar curvature is anti-self-dual if $det W^{+}\geq 0$. This generalizes results by Guan [11] and Polombo [17]. 

\vspace{20pt}

\section{\large\textbf{K\"ahler-Einstein surfaces and holomorphic sectional curvature}}\label{S:Intro}
In this section, we classify compact K\"ahler-Einstein surfaces with certain holomorphic sectional curvature conditions. 
Let $(M, g)$ be a riemannian manifold and $\nabla$ be the Levi-Civita connection. 
The curvature tensor is defined by 
\[R(X, Y)Z:=\nabla_{X}\nabla_{Y}Z-\nabla_{Y}\nabla_{X}Z-\nabla_{[X, Y]}Z.\]
The symmetry properties of the curvature tensor implies $R$ defines a symmetric bilinear map
 $R:\Lambda^{2}M\times\Lambda^{2}M\to\mathbb{R}$ such that 
\[R(X\wedge Y, Z\wedge W):=g(R(X, Y)W, Z),\]
where $\Lambda^{2}M$ is the set of bivectors. 
Then we have 
\[g(\mathfrak{R}(X\wedge Y, V\wedge W)=R(X\wedge Y, V\wedge W).\]

Let $(M, g, \omega)$ be a K\"ahler surface and $\{P, P^{\perp}\}$ be orthogonal holomorphic planes such that 
$P=\{X, JX\} $, $P^{\perp}=\{Y, JY\}$, where $X, Y$ are unit tangent vectors such that $X\perp Y, X\perp JY$. 
Then the orthogonal holomorphic bisectional curvature is defined by 
\[B(X, Y):=R(X\wedge JX, Y\wedge JY)\]
and the holomorphic sectional curvature by 
\[H(X, JX):=R(X\wedge JX, X\wedge JX).\]
Let $(M, g)$ be a riemannian four-manifold and let $\{P, P^{\perp}\}$ be the orthogonal planes. 
The biorthogonal curvature is defined by 
\[K^{\perp}(P, P^{\perp}):=\frac{K(P)+K(P^{\perp})}{2}\]
where $K(P):=R(X\wedge Y, X\wedge Y)$ is the sectional curvature on $P$ for unit tangent vectors $X, Y$ such that $P=\{X, Y\}$.

Let $(M, g, J)$ be an almost-Hermitian four-manifold. Then $g(X, Y)=g(JX, JY)$ for tangent vectors $X, Y\in TM$. 
Let $\omega(X, Y):=g(JX, Y)$. If $d\omega=0$, then $(M, g, \omega)$ is almost-K\"ahler. 
If $J$ is integrable and $d\omega=0$, then $(M, g, \omega)$ is K\"ahler. 
We note that $\nabla\omega=0$ for a K\"ahler metric, where $\nabla$ is the Levi-Civita connection. 
The following was shown in [1]. 
\newtheorem{Lemma}{Lemma}
\begin{Lemma}
Let $(M, g, J)$ be an almost-Hermitian four-manifolds. 
Then given an anti-self-dual 2-form $\phi$ with length $\sqrt{2}$, 
there exists an orthonormal basis $\{X, JX, Y, JY\}$ such that $\phi=X\wedge JX-Y\wedge JY$.
\end{Lemma}
\begin{proof}
We follow the proof given in [1]. 
Let $\phi$ be an anti-self-dual 2-form. 
Then an almost-complex structure $I$ s defined by 
\[g(IX, Y)=\frac{\sqrt{2}}{||\phi||}\phi(X, Y).\]
Similarly, $J$ correspond to a self-dual 2-form $\psi$. 
Since $\left<\phi, \psi\right>=0$, we have $\left<I, J\right>=0$.
Then we have 
\[\Sigma_{i=1}^{4}\left<I(e_{i}), J(e_{i})\right>=0.\]
If $\left<I(e_{i}), J(e_{i})\right>\neq 0$ for all $i$, then we can suppose
$\left<I(e_{1}), J(e_{1})\right>>0$ and $\left<I(e_{2}), J(e_{2})\right><0$.
Then there exists $t\in(0, 1)$ such that $\left<I(te_{1}+(1-t)e_{2}), J(te_{1}+(1-t)(e_{2})\right>=0$
Thus, there exists $V\neq 0$ such that $\left<I(V), J(V)\right>=0$. 
We note that from this, we get $\left<V, IJ(V)\right>=0$.
We consider $\{V, I(V), J(V), IJ(V)\}$, which is negatively oriented and $\{V, J(V), I(V), JI(V)\}$, which is positively oriented. 
Then we get $IJ(V)=JI(V)$. 
Then we have
\[J(V-IJV)=JV+IV=I(V-IJV).\]
Since $V\perp IJV$, we get $0\neq E_{1}:=V-IJV$ such that $JE_{1}=IE_{1}$. 
We also get 
\[J(V+IJV)=JV-IV=I(-V-IJV).\]
Thus, we have $0\neq E_{3}:=V+IJV$ such that $JE_{3}=-IE_{3}$. 
Since
$\left<V-IJV, V+IJV\right>=0$, we have $E_{1}\perp E_{3}$. 
Moreover, we have $\left<JE_{1}, E_{3}\right>=\left<JV-JIJV, V+IJV\right>=0$. 
We normalize $E_{1}$ and $E_{3}$ and we use the same notation for the dual 1-forms. 
Then, we get 
\[\phi=E_{1}\wedge IE_{1}+E_{3}\wedge IE_{3}=E_{1}\wedge JE_{1}-E_{3}\wedge JE_{3}.\]
\end{proof}

\begin{Lemma}
Let $(M, g, \omega, J)$ be a K\"ahler surface and $\lambda_{i}$ be the eigenvalues of $W^{-}$ such that $\lambda_{1}\leq \lambda_{2}\leq \lambda_{3}$. 
Let $B_{max}$ be the maximum of the orthogonal holomorphic bisectional curvature at each point
and $B_{min}$ be the minimum of the orthogonal holomorphic bisectional curvature at each point. 
Then we have
\[\frac{s}{6}-\lambda_{3}=2B_{min}, \hspace{10pt} \frac{s}{6}-\lambda_{1}=2B_{max}, \]
where $s$ is the scalar curvature.

\end{Lemma}
\begin{proof}
By Lemma 1, given an anti-self-dual 2-form $\phi$ with length $\sqrt{2}$,  there exists an orthonormal basis $\{E_{1}, JE_{1}, E_{3}, JE_{3}\}$ such that 
\[\phi=E_{1}\wedge JE_{1}-E_{3}\wedge JE_{3}.\]
Moreover, 
\[\left(\frac{s}{6}-W^{-}\right)(\phi, \phi)=R_{1313}+R_{1414}+R_{2323}+R_{2424}+2R_{1234},\]
where $R_{ijkl}=R(E_{i}\wedge E_{j}, E_{k}\wedge E_{l})$
and $E_{2}=JE_{1}, E_{4}=JE_{3}$.
Using curvature identities of K\"ahler manifolds and the First Bianchi identity, we have 
\[\left(\frac{s}{6}-W^{-}\right)(\phi, \phi)=4R_{1234}=4R(E_{1}\wedge JE_{1}, E_{3}\wedge JE_{3})=4B(E_{1}, E_{3}).\]
Let $\phi$ is an eigenvector of $W^{-}$ with eigenvalue $\lambda_{3}$. Then it follows that 
\[\frac{s}{6}-\lambda_{3}\geq 2B_{min}.\]
On the other hand, if $E_{1}, E_{3}$ are tangent vectors such that $R(E_{1}\wedge JE_{1}, E_{3}\wedge JE_{3})=B_{min}$, 
then 
\[\frac{s}{6}-\lambda_{3}\leq\frac{1}{2}\left(\frac{s}{6}-W^{-}\right)(\psi, \psi)=2B_{min},\]
where $\psi=E_{1}\wedge JE_{1}-E_{3}\wedge JE_{3}$.
Thus, it follows that 
\[\frac{s}{6}-\lambda_{3}= 2B_{min}.\]
Similarly, we have $\frac{s}{6}-\lambda_{1}=2B_{max}$.
\end{proof}

\begin{Lemma}
Let $(M, g, \omega, J)$ be a K\"ahler-Einstein surface with the scalar curvature $s$
and $\lambda_{i}$ be the eigenvalues of $W^{-}$ such that $\lambda_{1}\leq \lambda_{2}\leq \lambda_{3}$.
Then we have 
\[H_{max}=\frac{s}{6}+\frac{\lambda_{3}}{2}, \hspace{10pt} H_{min}=\frac{s}{6}+\frac{\lambda_{1}}{2}, \hspace{10pt} H_{av}=\frac{s}{6},\]
where $H_{max}(H_{min})$ is the maximum(minimum) of the holomorphic sectional curvature at each point respectively
and $H_{av}$ is the average of the holomorphic sectional curvature at each point. 
\end{Lemma}
\begin{proof}
Let $\{X, JX, Y, JY\}$ is an orthonormal basis.  
The First Bianchi identity and curvature identities for K\"ahler manifolds give 
\[B(X, Y):=R(X\wedge JX, Y\wedge JY)=R(X\wedge Y, X\wedge Y)+R(X\wedge JY, X\wedge JY).\]
On the other hand, 
\[Ric(X, X)=R(X\wedge JX, X\wedge JX)+R(X\wedge Y, X\wedge Y)+R(X\wedge JY, X\wedge JY).\]
Since $(M, g)$ is Einstein, we have 
\[Ric(X, X)=\frac{s}{4}.\]
It follows that 
\[\frac{s}{4}=H(X, JX)+B(X, Y)\]
and therefore, for $X$ such that $H_{max}=H(X, JX)$, we have 
\[\frac{s}{4}-H_{max}=B(X, Y)\geq B_{min}.\]
On the other hand, for $X, Y$ such that $B_{min}=B(X, Y)$, we have 
\[\frac{s}{4}-B_{min}=H(X, JX)\leq H_{max}.\]
Thus, we get 
\[\frac{s}{4}-B_{min}=H_{max}.\]
Then by Lemma 2, it follows that 
\[\frac{s}{6}-\lambda_{3}=2\left(\frac{s}{4}-H_{max}\right).\]
Thus, 
\[H_{max}=\frac{s}{6}+\frac{\lambda_{3}}{2}.\]
Similarly, we have 
\[H_{min}=\frac{s}{6}+\frac{\lambda_{1}}{2}.\]
\end{proof}

From the intersection pairing of the de Rham cohomology $H^{2}(M, \mathbb{R})$ on a smooth, compact four-manifold $M$, 
the diagonal matrix is given. 
We denote by $b_{\pm}$ the number of $(\pm 1)$'s of this matrix. 
Let $b_{i}$ be the i-th betti number. 
Then we have $b_{2}=b_{+}+b_{-}$ and $\tau(M)=b_{+}-b_{-}$, 
where $\tau(M)$ is the signature. 
The Euler characteristic $\chi(M)$ is given by $\chi(M)=2-b_{1}+b_{+}+b_{-}-b_{3}$.  

\begin{Lemma}
Let $(M, g, \omega)$ be a compact K\"ahler-Einstein surface with $b_{-}=0$. 
Then $(M, g, \omega)$ has constant holomorphic sectional curvature $\frac{s}{6}$, where $s$ is the scalar curvature. 
\end{Lemma}
\begin{proof}
We have 
\[(\chi-3\tau)(M)=\frac{1}{8\pi^{2}}\int_{M}\left(\frac{s^{2}}{24}-|W^{+}|^{2}+3|W^{-}|^{2}-\frac{|\mathring{ric}|^{2}}{2}\right)d\mu.\]
Since $\int_{M}\frac{s^{2}}{24}d\mu=\int_{M}|W^{+}|^{2}d\mu$ for a K\"ahler surface, we have 
$\chi\geq3\tau$ for a compact K\"ahler-Einstein surface.
From
\[(\chi-3\tau)(M)=2+b_{+}-b_{1}-b_{3}-3b_{+}\geq 0,\]
we have $b_{+}=1$ and $\chi(M)=3\tau(M)$. 
Therefore, $(M, g, \omega)$ is self-dual K\"ahler-Einstein. By Lemma 3, $(M, g, \omega)$ has constant holomorphic sectional curvature. 
\end{proof}

Gray showed that a compact K\"ahler manifold with nonnegative sectional curvature is locally symmetric 
if the scalar curvature is constant [10]. 
Below we use the orthogonal holomorphic bisectional curvature instead of the sectional curvature incase of K\"ahler surfaces. 

\newtheorem{Proposition}{Proposition}
\begin{Proposition}
Let $(M, g, \omega)$ be a compact K\"ahler surface with nonnegative orthogonal holomorphic bisectional curvature.
Suppose the scalar curvature is positive constant.
Then $(M, g, \omega)$ is locally symmetric. 
\end{Proposition}
\begin{proof}
By Lemma 2, we get 
\[\frac{s}{6}-\lambda_{3}=2B_{min}\geq 0.\]
If $b_{-}\neq 0$, then the following Weitzenb\"ock formula for an anti-self-dual 2-form implies
there exists a parallel anti-self-dual 2-form
\[\Delta\alpha=\nabla^{*}\nabla\alpha-2W^{-}(\alpha)+\frac{s}{3}\alpha.\]
Thus, $(\overline{M}, g)$ is K\"ahler. 
Since $(M, g)$ is locally a product metric and the scalar curvature is constant, $\nabla W^{\pm}=0$. 
If  $(M, g, \omega)$ is not Einstein,
then $\rho-\frac{s}{4}\omega$ is a non-trivial parallel 2-form
since $\rho-\frac{s}{4}\omega$ is an anti-self-dual harmonic 2-form. 
Since $\nabla\omega=0$, $\nabla\rho=0$. 
Thus, $(M, g, \omega)$ is locally symmetric. 

Suppose $b_{-}=0$. Following [15],  geometric genus $h^{2, 0}$ of a compact K\"ahler surface of positive scalar curvature vanishes [20]. 
Since $b_{+}=1+2h^{2, 0}$ [15] for any compact K\"ahler surface, we get $b_{+}=1$. 
Let $\rho$ be the Ricci-form. 
Since $\left<\rho, \omega\right>\omega=\frac{s}{4}\omega$, 
$d\rho^{+}=0$ and therefore, $d*\rho=0$. 
Thus, $\rho$ is a harmonic 2-form. 
Since $b_{2}=1$, $(M, g, \omega)$ is K\"aher-Einstein. 
By Lemma 3, $(M, g, \omega)$ has constant holomorphic sectional curvature. 
\end{proof}

\newtheorem{Theorem}{Theorem}
\begin{Theorem}
(Berger [4])
Let $(M, g, \omega)$ be a K\"ahler manifold of real dimension $n=2N$ and $H$ be the holomorphic sectional curvature. 
Then for all $p\in M$, we have 
\[\int_{S_{p}}Hd\sigma=\frac{s}{N(N+1)}vol(\mathbb{S}^{n-1}).\]
\end{Theorem}
From this, we get
\[H_{av}(p)=\frac{s}{6}\]
for a K\"ahler surface. 

\begin{Theorem}
(Hall and Murphy [13])
Let $(M, g, \omega, J)$ be an almost-Hermitian manifold of real dimension $n=2N$.
Then for all $p\in M$, 
\[\int_{S_{p}}Hd\sigma=\frac{s+3s^{*}}{4N(N+1)}vol(\mathbb{S}^{n-1}),\]
where $H$ is the holomorphic sectional curvature and $s^{*}=2R(\omega, \omega)$. 
\end{Theorem}
From this, for an almost-Hermitian four-manifold, it follows that 
\[H_{av}(p):=\frac{1}{vol(\mathbb{S}^{3})}\int_{S_{p}}Hd\sigma=\frac{s+3s^{*}}{24}.\]

Let $(M, g, \omega)$ be a K\"ahler-Einstein manifold with complex-dimension $N$.  
Then Berger showed $H\leq\frac{s}{N(N+1)}$ if and only if $H=\frac{s}{N(N+1)}$ [4]. 
In case of complex-dimension 2, we show this using eigenvalues of $W^{-}$. 

\begin{Theorem}
(Berger [4])
Let $(M, g, \omega)$ be a K\"ahler-Einstein surface and let $H$ be the holomorphic sectional curvature. 
Suppose $H\leq\frac{s}{6}$ or $H\geq\frac{s}{6}$. 
Then $(M, g, \omega)$ has constant holomorphic sectional curvature $\frac{s}{6}$,
where $s$ is the scalar curvature. 

\end{Theorem}
\begin{proof}
Suppose $H\leq\frac{s}{6}$. 
Then by Lemma 3, we have 
\[H_{max}=\frac{s}{6}+\frac{\lambda_{3}}{2}\leq\frac{s}{6}.\]
From this, we get $\lambda_{3}\leq 0$. Since $\lambda_{1}+\lambda_{2}+\lambda_{3}=0$
and $\lambda_{1}\leq\lambda_{2}\leq\lambda_{3}$, 
we get $\lambda_{1}=\lambda_{2}=\lambda_{3}=0$. 
Thus, $H=\frac{s}{6}$.
\end{proof}

Hall and Murphy showed the following result when $H_{av}=\frac{s}{3}H_{max}$ [13].
 \begin{Theorem}
Let $(M, g, \omega, J)$ be a compact K\"ahler-Einstein surface of nonzero scalar curvature.  
Suppose $H_{av}\geq\frac{2}{3}H_{max}$, where $H_{av}$ is the average of the holomorphic sectional curvature 
and $H_{max}$ is the maximum of the holomorphic sectional curvature. Then
\begin{itemize}
\item $(M, J)$ is biholomorphic to $\mathbb{CP}_{2}$
and $(M, g)$ is isometric to the Fubini-Study metric on $\mathbb{CP}_{2}$ up to rescaling; or
\item $(M, J)$ is biholomorphic to $\mathbb{CP}_{1}\times\mathbb{CP}_{1}$
and $(M, g)$ is isometric to the standard product metric of constant curvature on $\mathbb{CP}_{1}\times\mathbb{CP}_{1}$ up to rescaling. 
\end{itemize}
If $(M, g, \omega)$ has $H_{av}=\frac{2}{3}H_{max}$, then  
$(M, J)$ is biholomorphic to $\mathbb{CP}_{1}\times\mathbb{CP}_{1}$
and $(M, g)$ is isometric to the standard product metric of constant curvature on $\mathbb{CP}_{1}\times\mathbb{CP}_{1}$ up to rescaling. 
\end{Theorem}

\begin{proof}
Suppose $H_{av}\geq\frac{2}{3}H_{max}$.
Let $\lambda_{i}$ be the eigenvalues of $W^{-}$ such that $\lambda_{1}\leq\lambda_{2}\leq\lambda_{3}$. 
 By Lemma 3, we have 
\[\frac{s}{6}\geq\frac{2}{3}\left(\frac{s}{6}+\frac{\lambda_{3}}{2}\right).\]
From this, it follows that $\frac{s}{6}-\lambda_{3}\geq 0$ and the scalar curvature is positive. 
The following Weitzenb\"ock formula for anti-self-dual 2-forms, 
\[\Delta\omega=\nabla^{*}\nabla\omega-2W^{-}(\omega, \cdot)+\frac{s}{3}\omega,\]
implies either $b_{-}=0$ or there exists a parallel anti-self-dual 2-form. 
Suppose $b_{-}=0$. Then by Lemma 4, 
the result follows. 
Suppose $b_{-}\neq 0$. Then there exists a parallel anti-self-dual 2-form. 
Thus, $(\overline{M}, g)$ is K\"ahler, where $\overline{M}$ is $M$ with the opposite orientation. 
Since $(M, J)$, $\overline{M}$ are Fano surfaces and $\mathbb{CP}_{2}\#\overline{\mathbb{CP}_{2}}$ does not admit a K\"ahler-Einstein metric [16], 
$(M, g, J)$ is $\mathbb{CP}_{1}\times\mathbb{CP}_{1}$ with the standard metric ([8], p.125).

Suppose $H_{av}=\frac{2}{3}H_{max}$. Then we have $\frac{s}{6}=\lambda_{3}$. 
Since $s$ is nonzero, $\lambda_{3}>0$. In particular, $W^{-}\not\equiv0$.  
\end{proof}

\begin{Theorem}
Let $(M, g, \omega)$ be a compact K\"ahler-Einstein surface with nonzero scalar curvature and $H_{max}\leq 2H_{min}$, 
where $H_{max}(H_{min})$ is the maximum(minimum) of the holomorphic sectional curvature. Then
\begin{itemize}
\item $(M, J)$ is biholomorphic to $\mathbb{CP}_{2}$
and $(M, g)$ is isometric to the Fubini-Study metric on $\mathbb{CP}_{2}$ up to rescaling; or
\item $(M, J)$ is biholomorphic to $\mathbb{CP}_{1}\times\mathbb{CP}_{1}$
and $(M, g)$ is isometric to the standard product metric of constant curvature on $\mathbb{CP}_{1}\times\mathbb{CP}_{1}$ up to rescaling. 
\end{itemize}
If $(M, g, \omega)$ has $H_{max}=2H_{min}$, then  
$(M, J)$ is biholomorphic to $\mathbb{CP}_{1}\times\mathbb{CP}_{1}$
and $(M, g)$ is isometric to the standard product metric of constant curvature on $\mathbb{CP}_{1}\times\mathbb{CP}_{1}$ up to rescaling. 
\end{Theorem}
\begin{proof}
Let $\lambda_{i}$ be the eigenvalues of $W^{-}$ such that $\lambda_{1}\leq\lambda_{2}\leq\lambda_{3}$. 
By Lemma 3, it follows that 
\[\frac{s}{6}+\frac{\lambda_{3}}{2}\leq 2\left(\frac{s}{6}+\frac{\lambda_{1}}{2}\right)\leq \frac{s}{3}-\frac{\lambda_{3}}{2}\]
using $2\lambda_{1}\leq-\lambda_{3}$. 
It follows that $\frac{s}{6}-\lambda_{3}\geq 0$.
\end{proof}

\begin{Theorem}
Let $(M, g, \omega)$ be a compact K\"ahler-Einstein surface of nonzero scalar curvature. 
Suppose $H_{av}\leq\frac{4}{3}H_{min}$, where $H_{av}$ is the average of the holomorphic sectional curvature 
and $H_{max}$ is the maximum of the holomorphic sectional curvature. Then
\begin{itemize}
\item $(M, J)$ is biholomorphic to $\mathbb{CP}_{2}$
and $(M, g)$ is isometric to the Fubini-Study metric on $\mathbb{CP}_{2}$ up to rescaling; or
\item $(M, J)$ is biholomorphic to $\mathbb{CP}_{1}\times\mathbb{CP}_{1}$
and $(M, g)$ is isometric to the standard product metric of constant curvature on $\mathbb{CP}_{1}\times\mathbb{CP}_{1}$ up to rescaling. 
\end{itemize}
\end{Theorem}
\begin{proof}
Let $\lambda_{i}$ be the eigenvalues of $W^{-}$ such that $\lambda_{1}\leq\lambda_{2}\leq\lambda_{3}$. 
From  $H_{av}\leq\frac{4}{3}H_{min}$ and Lemma 3, we get 
\[-\frac{s}{12}\leq \lambda_{1}.\] 
Since $\lambda_{1}\leq 0$, we get $s>0$. 
Then we have [12]
\[\frac{1}{\sqrt{6}}|W^{-}|\leq -\lambda_{1}\leq\frac{s}{12}.\]
Since $\lambda_{1}+\lambda_{2}+\lambda_{3}=0$, 
we have 
\[|W^{-}|^{2}=(\lambda_{1})^{2}+(\lambda_{2})^{2}+(\lambda_{3})^{2}+
(\lambda_{1}-\lambda_{2}-\lambda_{3})(\lambda_{1}+\lambda_{2}+\lambda_{3})\]
\[=2((\lambda_{1})^{2}-\lambda_{2}\lambda_{3}).\]
If $\lambda_{2}\geq 0$, then $|W^{-}|^{2}\leq 2(\lambda_{1})^{2}$.
Suppose $\lambda_{2}\leq 0$. 
Since $-\lambda_{2}\leq -\lambda_{1}$ and $\lambda_{3}\leq-2\lambda_{1}$, it follows that
\[|W^{-}|^{2}=2((\lambda_{1})^{2}-\lambda_{2}\lambda_{3})\leq 2((\lambda_{1})^{2}+2(\lambda_{1})^{2})\leq 6(\lambda_{1})^{2}.\]
Moreover, the equality holds if and only if $\lambda_{1}=\lambda_{2}$ and $\lambda_{3}+2\lambda_{1}=0$. 
From this, it follows that 
\[\int_{M}|W^{-}|^{2}d\mu\leq \int_{M}\frac{s^{2}}{24}d\mu.\]
On the other hand, a compact, oriented Einstein four-manifold with positive scalar curvature and $W^{-}\not\equiv 0$ has 
\[\int_{M}|W^{-}|^{2}d\mu\geq\int_{M}\frac{s^{2}}{24}d\mu,\]
with equality if and only if $\nabla W^{-}=0$ by Gursky-LeBrun's Theorem below [12]. 
Suppose $W^{-}\not\equiv 0$. 
Since $\int_{M}|W^{-}|^{2}d\mu=\int_{M}\frac{s^{2}}{24}d\mu$, it follows that 
\[\frac{1}{\sqrt{6}}|W^{-}|=-\lambda_{1}=\frac{s}{12}.\]
Then we have $\lambda_{1}=-\frac{s}{12}, \lambda_{2}=-\frac{s}{12}, \lambda_{3}=\frac{s}{6}$.
Since $\int_{M}|W^{+}|^{2}d\mu=\int_{M}\frac{s^{2}}{24}d\mu$, $\tau=0$. 
Therefore, from the Weitzenb\"ock formula, we get a parallel anti-self-dual 2-form. 
Thus, $(M, g, \omega)$ is isometric to the standard product metric of constant curvature on $\mathbb{CP}_{1}\times\mathbb{CP}_{1}$.
\end{proof}

\begin{Theorem}
(Gursky-LeBrun)
Let $(M, g)$ be a compact, oriented Einstein four-manifolds with positive scalar curvature. 
Suppose $W^{+}\not\equiv 0$. 
Then 
\[\int_{M}|W^{+}|^{2}d\mu\geq\int_{M}\frac{s^{2}}{24}d\mu\]
with equality if and only if $\nabla W^{+}=0$. 
\end{Theorem}

\vspace{20pt}

\section{\large\textbf{K\"ahler surfaces with nonpositive biorthogonal curvature}}\label{S:Intro}
In this section, we show that a compact K\"ahler surface with nonpositive biorthogonal curvature has nonnegative signature. 
Moreover, we show that if a compact K\"ahler surface with nonpositive sectional curvature has zero signature, 
then the metric is locally a product.

\begin{Proposition}
Let $(M, g, \omega)$ be a K\"ahler surface. 
Let $K^{\perp}_{max}$ be the maximum of the biorthogonal curvature and $K^{\perp}_{min}$ be the minimum of the biorthogonal curvature. 
Let $B_{max}$ be the maximum of the orthogonal holomorphic bisectional curvature and $B_{min}$ be the minimum of the orthogonal holomorphic bisectional curvature. 
Let $\lambda_{i}^{-}$ be the eigenvalues of $W^{-}$ such that  $\lambda_{1}^{-}\leq\lambda_{2}^{-}\leq\lambda_{3}^{-}$. 
If the scalar curvature $s$ is nonnegative, then 
\[K^{\perp}_{min}=\frac{s}{24}+\frac{\lambda_{1}^{-}}{2} \hspace{10pt}and \hspace{10pt} K^{\perp}_{min}\leq \frac{B_{min}}{2}.\]
If the scalar curvature is nonpositive, then 
\[K^{\perp}_{max}=\frac{s}{24}+\frac{\lambda_{3}^{-}}{2} \hspace{10pt}and \hspace{10pt} K^{\perp}_{max}\geq\frac{B_{max}}{2}.\]
\end{Proposition}
\begin{proof}
Let $\lambda_{i}^{+}$ be the eigenvalues of $W^{+}$ such that $\lambda_{1}^{+}\leq\lambda_{2}^{+}\leq\lambda_{3}^{+}$.
The following formula was shown in [6]
\[K^{\perp}_{max}-\frac{s}{12}=\frac{\lambda_{3}^{+}+\lambda_{3}^{-}}{2},\]
where
\[K_{max}^{\perp}=max_{P\in T_{p}M}\left(\frac{K(P)+K(P^{\perp})}{2}\right) \]
and 
\[K^{\perp}_{min}-\frac{s}{12}=\frac{\lambda_{1}^{+}+\lambda_{1}^{-}}{2},\]
where
\[K_{min}^{\perp}=min_{P\in T_{p}M}\left(\frac{K(P)+K(P^{\perp})}{2}\right). \]
On a K\"ahler surface, we have 
\[W^{+}=\begin{pmatrix}-\frac{s}{12}&&\\&-\frac{s}{12}&\\&&\frac{s}{6}\end{pmatrix}.\]
If the scalar curvature is nonnegative, we have $\lambda_{3}^{+}=\frac{s}{6}$ and $\lambda_{1}^{+}=-\frac{s}{12}$. 
Using $2\lambda_{1}^{-}+\lambda_{3}^{-}\leq 0$ and Lemma 2, it follows that
\[K^{\perp}_{min}=\frac{s}{24}+\frac{\lambda_{1}^{-}}{2}\leq\frac{s}{24}-\frac{\lambda_{3}^{-}}{4}=\frac{B_{min}}{2}.\]
If the scalar curvature is nonpositive, we have $\lambda_{3}^{+}=-\frac{s}{12}$ and $\lambda_{1}^{+}=\frac{s}{6}$.
Since $2\lambda_{3}^{-}+\lambda_{1}^{-}\geq 0$, we have
\[K^{\perp}_{max}=\frac{s}{24}+\frac{\lambda_{3}^{-}}{2}\geq\frac{s}{24}-\frac{\lambda_{1}^{-}}{4}=\frac{B_{max}}{2}.\]
\end{proof}

\begin{Theorem}
Let $(M, g, \omega)$ be a compact K\"ahler surface with nonpositive biorthogonal curvature. 
Then $\tau(M)\geq 0$, where $\tau(M)$ is the signature of $M$. 
Moreover, if $\tau(M)=0$, then 
\[W^{-}=\begin{pmatrix}-\frac{s}{12}&&\\&-\frac{s}{12}&\\&&\frac{s}{6}\end{pmatrix}.\]
\end{Theorem}
\begin{proof}
We note that the scalar curvature is nonpositive since the biorthogonal curvature is nonpositive. 
Let $\lambda_{i}$ be the eigenvalues of $W^{-}$ such that $\lambda_{1}\leq \lambda_{2}\leq \lambda_{3}$.
Since 
\[K_{max}^{\perp}=\frac{s}{24}+\frac{\lambda_{3}}{2}\]
by Proposition 2, it follows that 
\[\lambda_{3}\leq -\frac{s}{12}.\]
Note that we have $\lambda_{1}+\lambda_{2}+\lambda_{3}=0$ and $|W^{-}|^{2}=(\lambda_{1})^{2}+(\lambda_{2})^{2}+(\lambda_{3})^{2}$.
The following formula was shown in [12]
\[\frac{1}{\sqrt{6}}|W^{-}|\leq\lambda_{3}.\]
Moreover, the equality holds if and only if $\lambda_{1}+2\lambda_{3}=0$ and $\lambda_{2}=\lambda_{3}$. 
Since $\lambda_{3}\leq-\frac{s}{12}$, it follows that 
\[\frac{1}{\sqrt{6}}|W^{-}|\leq-\frac{s}{12}.\]
Then we have 
\[|W^{-}|^{2}\leq\frac{s^{2}}{24}.\]

For a K\"ahler surface, we have $|W^{+}|^{2}= \frac{s^{2}}{24}$.
Then the signature formula gives 
\[\tau(M)=\frac{1}{12\pi^{2}}\int_{M}(|W^{+}|^{2}-|W^{-}|^{2})d\mu\geq 0.\]
Suppose $\tau=0$. Then $|W^{-}|^{2}=\frac{s^{2}}{24}$ and therefore, $\lambda_{1}=\frac{s}{6}, \lambda_{2}=-\frac{s}{12}, \lambda_{3}=-\frac{s}{12}$. 
\end{proof}

\newtheorem{Corollary}{Corollary}
\begin{Corollary}
A compact K\"ahler surface with negative biorthogonal curvature has positive signature.
\end{Corollary}
\begin{proof}
By Theorem 8, $\tau(M)\geq 0$. Moreover, if $\tau(M)=0$, then we have 
\[W^{-}=\begin{pmatrix}-\frac{s}{12}&&\\&-\frac{s}{12}&\\&&\frac{s}{6}\end{pmatrix}.\]
It follows that 
\[K^{\perp}_{max}=\frac{s}{12}+\frac{\lambda_{3}^{+}+\lambda_{3}^{-}}{2}=0,\]
which is a contradiction.  
\end{proof}

For a compact K\"ahler surface with nonpositive sectional curvature and zero signature, 
we follow the method given by Zheng [21]
and get a result under a weaker hypothesis. 
\begin{Lemma}
Let $(M, g)$ be an almost-Hermitian four-manifold. 
Let $\alpha_{i}^{\pm}$ be the local frames of $\Lambda^{\pm}$ respectively, 
where $\Lambda^{\pm}$ is a bundle of self-dual(anti-self-dual) 2-forms and suppose $\alpha_{1}^{+}=\frac{\omega}{\sqrt{2}}$.
Then there exists an orthonormal coframe $\{\theta_{1}, \theta_{2}, \theta_{3}, \theta_{4}\}$ such that
\[\alpha_{1}^{\pm}=\frac{\theta_{1}\wedge\theta_{2}\pm\theta_{3}\wedge\theta_{4}}{\sqrt{2}}, \]
\[\alpha_{2}^{\pm}=\frac{\theta_{1}\wedge\theta_{3}\pm\theta_{4}\wedge\theta_{2}}{\sqrt{2}}, \]
\[\alpha_{3}^{\pm}=\frac{\theta_{1}\wedge\theta_{4}\pm\theta_{2}\wedge\theta_{3}}{\sqrt{2}}. \]
Moreover, 
\[\theta_{2}(Je_{1})=1, \theta_{4}(Je_{3})=1,\]
where $\{e_{1}, e_{3}\}$ is the dual to $\{\theta_{1}, \theta_{3}\}$.
\end{Lemma}
\begin{proof}
We note that 
\[(\alpha_{i}^{+}+\alpha_{i}^{-})\wedge(\alpha_{j}^{+}+\alpha_{j}^{-})=0.\]
From this, it follows that here exists an orthonormal coframe $\{\theta_{1}, \theta_{2}, \theta_{3}, \theta_{4}\}$ such that 
\[\alpha_{1}^{+}+\alpha_{1}^{-}=\sqrt{2}\theta_{1}\wedge\theta_{2}\]
\[\alpha_{2}^{+}+\alpha_{2}^{-}=\sqrt{2}\theta_{1}\wedge\theta_{3},\]
for 2-forms on a four-manifold. 
Since $*\alpha_{1}^{+}=\alpha_{1}$ and $*\alpha_{1}^{-}=-\alpha_{1}$, we have 
\[\alpha_{1}^{+}-\alpha_{1}^{-}=\sqrt{2}\theta_{3}\wedge\theta_{4}\]
\[\alpha_{2}^{+}-\alpha_{2}^{-}=\sqrt{2}\theta_{4}\wedge\theta_{2}.\]
Since $\alpha_{3}^{\pm}\perp\alpha_{1}^{\pm}, \alpha_{3}^{\pm}\perp\alpha_{2}^{\pm}$ 
and $\frac{\theta_{1}\wedge\theta_{2}\pm\theta_{3}\wedge\theta_{4}}{\sqrt{2}}\perp\alpha_{1}^{\pm}, \alpha_{2}^{\pm}$, 
\[\alpha_{3}^{\pm}=\frac{\theta_{1}\wedge\theta_{4}\pm\theta_{2}\wedge\theta_{3}}{\sqrt{2}}. \]
Since 
\[\omega(X, JY)=g(X, Y),\]
we have 
\[(\theta_{1}\wedge\theta_{2}+\theta_{3}\wedge\theta_{4})(e_{1}, Je_{1})=g(e_{1}, e_{1})=1.\]
Thus, $\theta_{2}(Je_{1})=1$. 
Similarly, from
\[(\theta_{1}\wedge\theta_{2}+\theta_{3}\wedge\theta_{4})(e_{3}, Je_{3})=g(e_{3}, e_{3})=1,\]
we get $\theta_{4}(Je_{3})=1$. 
\end{proof}

\begin{Theorem}
Let $(M, g, \omega)$ be a compact K\"ahler surface with nonpositive sectional curvature.  
Then $\tau(M)\geq0$, where $\tau(M)$ is the signature of $M$. Moreover, if $\tau(M)=0$, then $(M, g)$ is locally a product metric. 
\end{Theorem}
\begin{proof}
By Theorem 8, $\tau(M)\geq 0$. 
Moreover, if $\tau(M)=0$, we have
\[W^{-}=\begin{pmatrix}\frac{s}{6}&&\\&-\frac{s}{12}&\\&&-\frac{s}{12}\end{pmatrix}.\]
Let $p\in M$ and suppose the scalar curvature is negative at $p$. 
Then there exists a neighborhood $U$ of $p$ such that the scalar curvature is negative on $U$. 
Let $\alpha_{i}^{\pm}$ be the eigenvectors of $W^{\pm}$ with eigenvalues $\left(\frac{s}{6}, -\frac{s}{12}, -\frac{s}{12}\right)$. 
Since the distinct numbers of the eigenvalues is 2 when $s<0$, 
$\alpha_{i}^{\pm}$ are on smooth 2-forms on a neighborhood $V$ of $p$. 
Thus, on $U\cap V$, the scalar curvature is negative and $\alpha_{i}^{\pm}$ are orthonormal frame of $\Lambda^{\pm}$. 
Then by Lemma 5, 
there exists an orthonormal coframe $\{\theta_{1}, \theta_{2}, \theta_{3}, \theta_{4}\}$ such that 
\[\alpha_{1}^{\pm}=\frac{\theta_{1}\wedge\theta_{2}\pm\theta_{3}\wedge\theta_{4}}{\sqrt{2}}, \]
\[\alpha_{2}^{\pm}=\frac{\theta_{1}\wedge\theta_{3}\pm\theta_{4}\wedge\theta_{2}}{\sqrt{2}}, \]
\[\alpha_{3}^{\pm}=\frac{\theta_{1}\wedge\theta_{4}\pm\theta_{2}\wedge\theta_{3}}{\sqrt{2}} \]
and $\{e_{1}, Je_{1}, e_{3}, Je_{3}\}$ is the dual frame. 
Then we have 

\[ 
\mathfrak{R}=
\LARGE
\begin{pmatrix}

 \begin{array}{c|c}
\scriptscriptstyle{\tilde{\mathfrak{R}}}& \scriptscriptstyle{}\\
 \hline
 
 \hspace{10pt}
 
 \scriptscriptstyle{}& \scriptscriptstyle{\tilde{\mathfrak{R}}}\\
 \end{array}
 \end{pmatrix}
 \]
where 
\[\tilde{\mathfrak{R}}=\begin{pmatrix}\frac{s}{4}&&\\&0&\\&&0\end{pmatrix}.\]
Let $e_{2}:=Je_{1}$ and $e_{4}:=Je_{3}$.

From 
\[\mathfrak{R}(e_{1}\wedge e_{2}+e_{3}\wedge e_{4}, e_{1}\wedge e_{2}+e_{3}\wedge e_{4})=\frac{s}{2},\]
\[\mathfrak{R}(e_{1}\wedge e_{3}+e_{4}\wedge e_{2}, e_{1}\wedge e_{3}+e_{4}\wedge e_{2})=0,\]
\[\mathfrak{R}(e_{1}\wedge e_{4}+e_{2}\wedge e_{3}, e_{1}\wedge e_{4}+e_{2}\wedge e_{3})=0,\]
and
\[\mathfrak{R}(e_{1}\wedge e_{2}-e_{3}\wedge e_{4}, e_{1}\wedge e_{2}-e_{3}\wedge e_{4})=\frac{s}{2},\]
\[\mathfrak{R}(e_{1}\wedge e_{3}-e_{4}\wedge e_{2}, e_{1}\wedge e_{3}-e_{4}\wedge e_{2})=0,\]
\[\mathfrak{R}(e_{1}\wedge e_{4}-e_{2}\wedge e_{3}, e_{1}\wedge e_{4}-e_{2}\wedge e_{3})=0,\]
we get $R_{1234}=0$, $R_{1313}+R_{4242}=0$, $R_{1342}=0$, $R_{1414}+R_{2323}=0$, $R_{1423}=0$. 
From the K\"ahler identity, 
\[R_{ij13}=R_{ij24}, \hspace{5pt} R_{ij14}=-R_{ij23},\]
we have 
\[R_{1313}=R_{2424}=R_{1324}=R_{1414}=R_{2323}=R_{1423}=0.\]

From 
\[\mathfrak{R}(e_{1}\wedge e_{2}-e_{3}\wedge e_{4}, e_{1}\wedge e_{3}-e_{4}\wedge e_{2})=0,\]
\[\mathfrak{R}(e_{1}\wedge e_{2}-e_{3}\wedge e_{4}, e_{1}\wedge e_{4}- e_{2}\wedge e_{3})=0,\]
\[\mathfrak{R}(e_{1}\wedge e_{3}- e_{4}\wedge e_{2}, e_{1}\wedge e_{4}-e_{2}\wedge e_{3})=0,\]
and K\"ahler identity, we get
\[R_{1213}-R_{3413}=0, \hspace{5pt} R_{1214}-R_{3414}=0,  \hspace{5pt} R_{1314}-R_{4214}=0.\]

From 
\[\mathfrak{R}(e_{1}\wedge e_{2}+e_{3}\wedge e_{4}, e_{1}\wedge e_{3}-e_{4}\wedge e_{2})=0,\]
\[\mathfrak{R}(e_{1}\wedge e_{2}+e_{3}\wedge e_{4}, e_{1}\wedge e_{4}- e_{2}\wedge e_{3})=0,\]
\[\mathfrak{R}(e_{1}\wedge e_{3}+e_{4}\wedge e_{2}, e_{1}\wedge e_{4}-e_{2}\wedge e_{3})=0,\]
and K\"ahler identity, we get
\[R_{1213}+R_{3413}=0, \hspace{5pt} R_{1214}+R_{3414}=0,  \hspace{5pt} R_{1314}+R_{4214}=0.\]
Therefore, we have 
\[R_{1213}=R_{3413}=R_{1214}=R_{3414}=R_{1314}=R_{4214}=0.\]
Let 
\[\Omega=\begin{pmatrix}\Omega_{1}^{1}&\Omega_{2}^{1}&\Omega_{3}^{1}&\Omega_{4}^{1}\\
\Omega_{1}^{2}&\Omega_{2}^{2}&\Omega_{3}^{2}&\Omega_{4}^{2}\\
\Omega_{1}^{3}&\Omega_{2}^{3}&\Omega_{3}^{3}&\Omega_{4}^{3}\\
\Omega_{1}^{4}&\Omega_{2}^{4}&\Omega_{3}^{4}&\Omega_{4}^{4}\\
\end{pmatrix},\]
where 2-forms $\Omega^{i}_{j}$ are defined by 
\[\Omega_{j}^{i}=\sum_{k<l}R_{ijkl}\theta^{k}\wedge\theta^{l}.\]
Then we have 
\[\Omega_{1}^{3}=R_{3112}\theta^{1}\wedge\theta^{2}+R_{3113}\theta^{1}\wedge\theta^{3}+R_{3114}\theta^{1}\wedge\theta^{4}\]
\[+R_{3123}\theta^{2}\wedge\theta^{3}+R_{3124}\theta^{2}\wedge\theta^{4}+R_{3134}\theta^{3}\wedge\theta^{4},\]
Since 
\[R_{3112}=R_{3113}=R_{3114}=R_{3124}=R_{3413}=0, R_{3123}=-R_{3114}=0,\]
$\Omega_{1}^{3}=0$. Similarly, we have $\Omega_{2}^{3}=\Omega_{1}^{4}=\Omega_{2}^{4}=0$. 
Since $\Omega_{i}^{j}=-\Omega_{j}^{i}$, it follows that 
\[\Omega=\begin{pmatrix}&\Omega_{2}^{1}&&\\
-\Omega_{2}^{1}&&&\\
&&&-\Omega_{3}^{4}\\
&&\Omega_{3}^{4}&\\
\end{pmatrix}.\]
Let $\omega_{j}^{i}$ be the connection 1-forms such that 
\[\omega=\begin{pmatrix}\omega_{1}^{1}&\omega_{2}^{1}&\omega_{3}^{1}&\omega_{4}^{1}\\
\omega_{1}^{2}&\omega_{2}^{2}&\omega_{3}^{2}&\omega_{4}^{2}\\
\omega_{1}^{3}&\omega_{2}^{3}&\omega_{3}^{3}&\omega_{4}^{3}\\
\omega_{1}^{4}&\omega_{2}^{4}&\omega_{3}^{4}&\omega_{4}^{4}\\
\end{pmatrix}.\]
We note that $\omega_{i}^{j}=-\omega_{j}^{i}$. Then we have 
\[d\theta^{i}=-\sum_{k}\omega^{i}_{k}\wedge\theta^{k}.\]
By the second Bianchi identity, we have 
\[d\Omega=\Omega\wedge\omega-\omega\wedge\Omega.\]
From this, it follows that 
\[\Omega_{2}^{1}\wedge\omega_{3}^{2}=\omega_{4}^{1}\wedge\Omega_{3}^{4}\]
\[\Omega_{2}^{1}\wedge\omega_{4}^{2}=\omega_{3}^{1}\wedge\Omega_{4}^{3}\]
\[\Omega_{1}^{2}\wedge\omega_{3}^{1}=\omega_{4}^{2}\wedge\Omega_{3}^{4}\]
\[\Omega_{1}^{2}\wedge\omega_{4}^{1}=\omega_{3}^{2}\wedge\Omega_{4}^{3}\]

Since $s=R_{1212}+R_{3434}$, $R_{1212}(p)$ or $R_{3434}(p)$ is less than $0$. 
Suppose $R_{1212}(p)<0$. 
Since $\Omega_{2}^{1}=R_{1212}\theta^{1}\wedge\theta^{2}$ and 
 $\Omega_{4}^{3}=R_{3434}\theta^{3}\wedge\theta^{4}$, 
 $\omega_{3}^{2}=0$ at $p$ since $\Omega_{2}^{1}\wedge\omega_{3}^{2}=\omega_{4}^{1}\wedge\Omega_{3}^{4}$.
Similarly, we get $\omega_{4}^{2}=\omega_{3}^{1}=\omega_{4}^{1}=0$ at $p$. 
If $R_{3434}(p)<0$, we also get $\omega_{3}^{2}=\omega_{4}^{2}=\omega_{3}^{1}=\omega_{4}^{1}=0$ at $p$.
Using this, at $p$, it follows that 
\[\sqrt{2}d\alpha_{1}^{-}=d\theta^{1}\wedge\theta^{2}-\theta^{1}\wedge d\theta^{2}-d\theta^{3}\wedge\theta^{4}+\theta^{3}\wedge d\theta^{4}\]
\[=-\omega_{3}^{1}\wedge\theta^{3}\wedge\theta^{2}-\omega_{4}^{1}\wedge\theta^{4}\wedge\theta^{2}-\omega_{3}^{2}\wedge\theta^{1}\wedge\theta^{3}-
\omega_{4}^{2}\wedge\theta^{1}\wedge\theta^{4}\]
\[+\omega_{1}^{3}\wedge\theta^{1}\wedge\theta^{4}-\omega_{2}^{3}\wedge\theta^{2}\wedge\theta^{4}+\omega_{1}^{4}\wedge\theta^{3}\wedge\theta^{1}+
\omega_{2}^{4}\wedge\theta^{3}\wedge\theta^{2}=0.\]
Thus, $\alpha_{1}^{-}$ is closed on the set where the scalar curvature is less than zero. 
Since $\alpha_{1}^{-}$ is anti-self-dual, we have 
\[\delta\alpha_{1}^{-}=-*d*\alpha_{1}^{-}=0.\]
Thus, $\alpha_{1}^{-}$ is harmonic on the set where the scalar curvature is less than zero. 
From the Weitzenb\"ock formula, we have 
\[\Delta\alpha_{1}^{-}=\nabla^{*}\nabla\alpha_{1}^{-}+\left(\frac{s}{6}-W^{-}\right)(\alpha_{1}^{-}).\]
Since $\alpha_{1}^{-}$ is an eigenvector of $W^{-}$ with the eigenvalue $\frac{s}{6}$, we get $\nabla^{*}\nabla\alpha_{1}^{-}=0$.
Since 
\[\left<\nabla^{*}\nabla\alpha_{1}^{-}, \alpha_{1}^{-}\right>=\frac{1}{2}\Delta|\alpha_{1}^{-}|^{2}+\left<\nabla\alpha_{1}^{-}, \nabla\alpha_{1}^{-}\right>\]
and $|\alpha_{1}^{-}|=1$, we get $\alpha_{1}^{-}$ is parallel. 
Since $\omega$ is parallel, $g$ is a product metric on the set where the scalar curvature is less than zero by the de Rahm decomposition Theorem. 
Suppose $s=0$ at $p$. Then the sectional curvature is zero. Thus, $g$ is flat on the set where the scalar curvature is zero. 
Thus, $(M, g)$ is locally a product metric. 

\end{proof}
 The same proof shows the following. 
 \begin{Theorem}
 Let $(M, g, \omega)$ be a compact K\"ahler surface with nonpositive biorthogonal curvature and negative scalar curvature. 
 Then $\tau(M)\geq 0$, where $\tau(M)$ is the signature of $M$. 
 Moreover, if $\tau(M)=0$, then $(M, g)$ is locally a product metric. 
  \end{Theorem}

\vspace{20pt}

\section{\large\textbf{Harmonic self-dual Weyl curvature with constant length}}\label{S:Intro}
Wu [19] and LeBrun classified compact riemannian four-manifolds with harmonic self-dual Weyl curvature
and $det W^{+}>0$. 
\begin{Theorem}
(LeBrun [15]) Let $(M, g)$ be a compact oriented four-manifold with $\delta W^{+}=0$. 
Let $\lambda_{i}$ be the eigenvalues of $W^{+}$ such that $\lambda_{1}\leq\lambda_{2}\leq\lambda_{3}$. 
Suppose $det W^{+}>0$. 
Then 
\begin{itemize}
\item $b_{+}(M)=1$ and $(M, g)$ is conformal to a K\"ahler metric with positive scalar curvature; or
\item $b_{+}(M)=0$ and the double cover of $(M, g)$ with the pull-back metric is conformal to a K\"ahler metic with positive scalar curvature. 

\end{itemize} 
\end{Theorem}

\newtheorem{Remark}{Remark}
\begin{Remark}
In Theorem 11, a weaker condition can be assumed such that 
$\lambda_{3}>0$ and $\lambda_{2}\leq 0$ instead of $det W^{+}> 0$ [15]. 
The same argument shows that $(M, g)$ is conformal to a K\"ahler metric with conformal factor $f^{-2}$ where $f:=(\lambda_{3})^{-1/3}$. 
Since 
\[W^{+}=\begin{pmatrix}-\frac{s}{12}&&\\&-\frac{s}{12}&\\&&\frac{s}{6}\end{pmatrix}\]
for a K\"ahler metric, 
we have 
\[W^{+}_{g}=f^{-2}\begin{pmatrix}-\frac{s}{12}&&\\&-\frac{s}{12}&\\&&\frac{s}{6}\end{pmatrix}.\]
Since $\lambda_{3}>0$ and $\lambda_{2}\leq 0$, it follows that $s>0$ and $\lambda_{2}<0$. 
In particular, $det W^{+}>0$. 
\end{Remark}

\begin{Remark}
We note that by changing the orientation, $W^{-}$ can be used instead of $W^{+}$ in Theorem 11. 

\end{Remark}

Using Wu [19] and LeBrun's Theorem  11, we show the following. 
Polombo showed this result by assuming $g'$ is K\"ahler-Einstein [17].
\begin{Theorem}
Let $(M, g)$ be a compact Einstein four-manifold with positive scalar curvature. 
Suppose there exists a K\"ahler metric $g'$ which is sufficiently close to $g$ in $C^{2}$-Topology. Then $(M, g)$ is K\"ahler-Einstein.
\end{Theorem}
\begin{proof}
Since $g'$ is sufficiently close to $g$, 
the scalar curvature of $g'$ is positive.
Since $g'$ is a K\"ahler metric with positive scalar curvature, $det W^{+}_{g'}>0$ and therefore, $det W^{+}_{g}>0$. 
Since $b_{+}\geq 1$, $(M, g)$ is conformal to a K\"ahler metric $h$ by Theorem 11. 
Compact oriented Einstein four-manifolds with positive scalar curvature and $W^{+}\not\equiv 0$ has 
\[\int_{M}|W^{+}_{g}|^{2}_{g}d\mu_{g}\geq\int_{M}\frac{s_{g}^{2}}{24}d\mu_{g},\]
with equality holds if and only if $\nabla W^{+}=0$ by Theorem 7. 
Suppose 
\[\int_{M}|W^{+}_{g}|^{2}_{g}d\mu_{g}-\int_{M}\frac{s_{g}^{2}}{24}d\mu_{g}>0.\]
Since $g'$ is sufficiently close to $g$ in $C^{2}$-Topology, 
\[\int_{M}|W^{+}_{g'}|^{2}_{g'}d\mu_{g'}-\int_{M}\frac{s_{g'}^{2}}{24}d\mu_{g'}>0,\]
which is a contradiction since $\int_{M}|W^{+}_{g'}|^{2}_{g'}d\mu_{g'}=\int_{M}\frac{s_{g'}^{2}}{24}d\mu_{g'}$
for a K\"ahler metric. 
Thus, $(M, g)$ is a compact Einstein four-manifold with positive scalar curvature such that 
$\int_{M}\frac{s^{2}}{24}d\mu=\int_{M}|W^{+}|^{2}d\mu$.
Thus, either $W^{+}=0$ or we get $\nabla W^{+}=0$. 
On the other hand, a compact anti-self-dual Einstein four-manifold with positive scalar curvature
is isometric to the Fubini-Study metric with the oppoiste orientation on $\overline{\mathbb{CP}_{2}}$ or the standard metric on $S^{4}$ up to rescaling [5], [9], [14].
Since they are not close to a K\"ahler metric, $(M, g)$ is not anti-self-dual and therefore, $\nabla W^{+}=0$. 
Since $(M, g)$ is conformal to a K\"ahler metric, $W^{+}$ has a degenerate spectrum. 
Since $det W^{+}_{g}>0$, $W^{+}_{g}$ is nowhere zero. 
Therefore, $(M, g)$ is K\"ahler  [7]. 
\end{proof}

\begin{Proposition}
Let $(M, g)$ be a compact, oriented four-manifold with nonnegative scalar curvature.
Suppose $\delta W^{+}=0$ and $|W^{+}|$ is constant. Then 
\begin{itemize}
\item $W^{+}\equiv 0$; or
\item $b_{+}(M)=1$ and $(M, g)$ is a K\"ahler metric with positive constant scalar curvature; or
\item $b_{+}(M)=0$ and the double cover of $(M, g)$ is a K\"ahler metric with positive constant scalar curvature. 
\end{itemize}

\end{Proposition}
\begin{proof}
Suppose $|W^{+}|\neq 0$. 
The following formula [5], 
\[\Delta|W^{+}|^{2}=-2|\nabla W^{+}|^{2}+36detW^{+}-s|W^{+}|^{2},\]
 implies $det W^{+}\geq0$.
 Since $|W^{+}|\neq 0$, $\lambda_{3}>0$. 
 Then by Theorem 11 and Remark 1, $W^{+}$ has a degenerate spectrum. 
 Since $|W^{+}|$ is constant, the eigenvalues are constant. 
  The result follows from Theorem 11 and [7]. 
 \end{proof}

\begin{Corollary}
Let $(M, g)$ be a compact, oriented four-manifold with nonnegative scalar curvature. 
Suppose $\delta W^{+}=0$ and $|W^{+}|$ is constant. 
If the scalar curvature vanishes at a point $p\in M$, then $(M, g)$ is anti-self-dual. 
\end{Corollary}

\begin{Proposition}
Let $(M, g)$ be a compact, oriented four-manifold with $\delta W^{+}=0$. 
Suppose the scalar curvature is nonnegative. 
If $\int_{M}det W^{+}d\mu\leq 0$, 
then $(M, g)$ is anti-self-dual. 

\end{Proposition}
\begin{proof}
From 
\[\Delta|W^{+}|^{2}=-2|\nabla W^{+}|^{2}+36det W^{+}-s|W^{+}|^{2}, \]
it follows that $\nabla W^{+}=0$ and $det W^{+}=0$.
Then
\[d\left<W^{+}, W^{+}\right>=2\left<\nabla W^{+}, W^{+}\right>=0.\]
Thus, $|W^{+}|$ is constant.
Suppose $|W^{+}|>0$. 
Then in particular, $\lambda_{3}>0$, where $\lambda_{i}$ are the eigenvalues of $W^{+}$ such that 
$\lambda_{1}\leq\lambda_{2}\leq\lambda_{3}$. 
Then by Theorem 11 and Remark 1, $det W^{+}>0$, which is a contradiction. 
Thus, $|W^{+}|=0$ and therefore, $(M, g)$ is anti-self-dual. 
\end{proof}

\begin{Corollary}
Let $(M, g, \omega)$ be a compact K\"ahler-Einstein surface with nonnegative scalar curvature. 
Suppose 
\[H_{av}-H_{min}\geq\frac{1}{2}(H_{max}-H_{min}).\]
Then  
\begin{itemize}
\item $(M, J)$ is biholomorphic to $\mathbb{CP}_{2}$ and $(M, g)$ is isometric to the Fubini-Study metric on $\mathbb{CP}_{2}$ up to rescaling; or
\item $(M, g, \omega)$ is locally flat-K\"ahler. 
\end{itemize}

\end{Corollary}
\begin{proof}
From $\lambda_{2}\geq 0$, it follows $det W^{-}\leq 0$. By Proposition 4, $(M, g)$ is self-dual. 
\end{proof}

\vspace{20pt}

\section{\large\textbf{Pinched self-dual Weyl curvature and nonpositive scalar curvature}}\label{S:Intro}
In this section, we show compact riemannian four-manifolds with harmonic self-dual Weyl curvature and nonpositive scalar curvature
is anti-self-dual if $det W^{+}\geq 0$.
\begin{Theorem}
(Polombo) Let $(M, g)$ be a compact oriented riemannian four-manifold with $\delta W^{+}=0$. 
Suppose the scalar curvature is nonpositive and 
\[-8\left(1-\frac{\sqrt{3}}{2}\right)\lambda_{1}\leq \lambda_{3}\leq -2\lambda_{1},\]
where $\lambda_{i}$ be the eigenvalues of $W^{+}$ such that $\lambda_{1}\leq\lambda_{2}\leq\lambda_{3}$. 
Then $(M, g)$ is anti-self-dual. 
\end{Theorem}

\begin{Theorem}
(B\"ar, Polombo) Let $(M, g)$ be a compact oriented four-manifold with $\delta W^{+}=0$.
Then the set $N=\{x\in M|\lambda_{1}=\lambda_{2}=\lambda_{3}=0\}$ is codimension $\geq 2$ in $M$ unless $(M, g)$ is anti-self-dual. 
\end{Theorem}

Let $(M, g)$ be an oriented four-manifold with $\delta W^{+}=0$. 
Let $M_{W}$ be the set where the number of distinct eigenvalues of $W^{+}$ is locally constant. 
Then $M_{W}$ is an open dense subset [7]. 
At a point $p\in M$, let $\lambda_{i}$ be the eigenvalues of $W^{+}$ such that $\lambda_{1}\leq\lambda_{2}\leq\lambda_{3}$
and $\{\rho, \eta, \varphi\}$ is a basis of orthogonal eigenvectors of length $\sqrt{2}$ so that 
\[W^{+}=\frac{1}{2}[\lambda_{1}\rho\otimes\rho+\lambda_{2}\eta\otimes\eta+\lambda_{3}\varphi\otimes\varphi].\]
Then in a neighborhood of $p\in M_{W}$, it may be assumed that $\lambda_{i}, \rho, \eta, \varphi$ are differentiable [7]. 
Then locally in a neighborhood of $p$, we have 
\[\nabla\rho=c\otimes\eta-b\otimes\varphi\]
\[\nabla\eta=-c\otimes\rho+a\otimes\varphi\]
\[\nabla\varphi=b\otimes\rho-a\otimes\eta\]
for some 1-forms $a, b, c$.

From $\delta W^{+}=0$, locally in a neighborhood of $p\in M_{W}$, we have [7] 
\[d\lambda_{1}=(\lambda_{1}-\lambda_{2})\varphi c+(\lambda_{1}-\lambda_{3})\eta b\]
\[d\lambda_{2}=(\lambda_{2}-\lambda_{1})\varphi c+(\lambda_{2}-\lambda_{3})\rho a\]
\[d\lambda_{3}=(\lambda_{3}-\lambda_{1})\eta b+(\lambda_{3}-\lambda_{2})\rho a\]
and 
\[\Delta \lambda_{1}=2\lambda_{1}^{2}+4\lambda_{2}\lambda_{3}-\frac{\lambda_{1}s}{2}+2(\lambda_{3}-\lambda_{1})|b|^{2}+2(\lambda_{2}-\lambda_{1})|c|^{2}\]
\[\Delta \lambda_{2}=2\lambda_{2}^{2}+4\lambda_{1}\lambda_{3}-\frac{\lambda_{2}s}{2}+2(\lambda_{1}-\lambda_{2})|c|^{2}+2(\lambda_{3}-\lambda_{2})|a|^{2}\]
\[\Delta \lambda_{3}=2\lambda_{3}^{2}+4\lambda_{1}\lambda_{2}-\frac{\lambda_{3}s}{2}+2(\lambda_{1}-\lambda_{3})|b|^{2}+2(\lambda_{2}-\lambda_{3})|a|^{2},\]
where $\Delta=\delta d$. 

We note that since $N$ has codimension $\geq 2$, $N\cap M_{W}=\varnothing$ unless $W^{+}\equiv 0$. 

\begin{Lemma}
Let $(M, g)$ be a compact oriented riemannian four-manifold with harmonic self-dual Weyl curvature. 
Suppose $\lambda_{1}+\lambda_{3}\geq 0$,  where $\lambda_{i}$ are the eigenvalues of $W^{+}$ such that $\lambda_{1}\leq\lambda_{2}\leq\lambda_{3}$.
Then $\lambda_{1}+\lambda_{3}>0$ on $M-N$ unless $(M, g)$ is anti-self-dual, 
where $N$ is the nonempty zero set of $W^{+}$. 
\end{Lemma}
\begin{proof}
Let $M_{W}$ be the set where the number of distinct eigenvalues of $W^{+}$ is locally constant. 
Suppose $\lambda_{2}=0$ at $p\in M_{W}$ and $W^{+}\not\equiv 0$. 
Since $\lambda_{2}\leq 0$, 
$\lambda_{2}$ has a local maximum $0$ at $p$. 
Then $\nabla\lambda_{2}=0$ at $p$ and $\Delta\lambda_{2}\geq 0$ at $p$. 
Since $p\in M_{W}$, using the formula for $\Delta\lambda_{2}$ and $\lambda_{2}=0$ at $p$, it follows that
\[\Delta\lambda_{2}=4\lambda_{1}\lambda_{3}+2\lambda_{1}|c|^{2}+2\lambda_{3}|a|^{2}\geq 0.\]
We have 
\[0=d\lambda_{2}=-\lambda_{1}\varphi c-\lambda_{3}\rho a\]
at $p$ and therefore 
\[\lambda_{1}^{2}|c|^{2}=\lambda_{3}^{2}|a|^{2}\]
at $p$.
We note that $|\rho a|=|a|, |\eta b|=|b|, |\varphi c|=|c|$.
Since $\lambda_{1}^{2}=\lambda_{3}^{2}$ at $p$, we have $|c|=|a|$ at $p$. 
Since $\lambda_{1}+\lambda_{3}=0$ at $p$, it follows that 
\[\Delta\lambda_{2}=4\lambda_{1}\lambda_{3}-2\lambda_{3}|c|^{2}+2\lambda_{3}|a|^{2}=4\lambda_{1}\lambda_{3}\leq 0.\] 
Thus, $\lambda_{1}\lambda_{3}=0$ at $p$, and therefore, $\lambda_{1}=\lambda_{2}=\lambda_{3}=0$ at $p$, which implies $p\in N$. 
On the other hand, $N\cap M_{W}=\varnothing$ unless $W^{+}\equiv 0$. 
Thus, $\lambda_{2}\neq0$ on $M_{W}$ and therefore, $\lambda_{1}+\lambda_{3}>0$ on $M_{W}$. 

Suppose $p\in M-(M_{W}\cup N)$ and $\lambda_{2}=0$ at $p$. 
If $W^{+}$ has three distinct eigenvalues at $p$, 
then near $p$, $W^{+}$ also has three distinct eigenvalues. 
Thus, $p\in M_{W}$. 
Therefore, if $p\in M-(M_{W}\cup N)$, then either $2\lambda_{1}+\lambda_{3}=$ or $\lambda_{1}+2\lambda_{3}=0$ at $p$. 
Since $\lambda_{1}+\lambda_{3}=0$ at $p$, we get $\lambda_{1}=\lambda_{2}=\lambda_{3}=0$ at $p$, 
which is a contradiction. 
Thus, if if $p\in M-(M_{W}\cup N)$, then $\lambda_{2}<0$ at $p$. 
\end{proof}

\begin{Corollary}
Let $(M, g)$ be a compact oriented riemannian four-manifold with harmonic self-dual Weyl curvature. 
Suppose $det W^{+}=0$. Then $(M, g)$ is anti-self-dual. 
\end{Corollary}
\begin{proof}
Suppose $(M, g)$ is not anti-self-dual. Then by Lemma 6, $\lambda_{2}<0$ on $M-N$.
If $\lambda_{1}=0$ or $\lambda_{3}=0$ at $p\in M-N$, then $\lambda_{2}=0$ at $p\in M-N$. 
Thus, $\lambda_{1}<0$ and $\lambda_{3}>0$ on $M-N$. 
Then $det W^{+}\neq 0$. 

\end{proof}

\begin{Theorem}
Let $(M, g)$ be a compact oriented four-manifold with nonpositive scalar curvature. 
Suppose $\delta W^{+}=0$ and $det W^{+}\geq 0$, equivalently, $\lambda_{1}+\lambda_{3}\geq 0$,
where $\lambda_{i}$ be the eigenvalues of $W^{+}$ such that $\lambda_{1}\leq\lambda_{2}\leq\lambda_{3}$.
Then $(M, g)$ is anti-self-dual.

\end{Theorem}
\begin{proof}
We follow the method by Guan [11].
Using Lemma 2, the subharmonic function $\Phi$ in [11] in K\"ahler-Einstein case, which is called Hong Cang Yang's function in [11], is given by 
\[\Phi=\lambda_{1}+\lambda_{3}.\]
So we use $\Phi=\lambda_{1}+\lambda_{3}=-\lambda_{2}$ for more general results. 
Following [11], [18], we consider $\Delta\Phi^{\lambda}$ for $\lambda=1/3$.
Then we have
\[\Delta\Phi^{\lambda}=\lambda\Phi^{\lambda-2}(\Phi\Delta\Phi+(1-\lambda)\Sigma|\nabla_{s}\Phi|^{2}),\]
where $\Delta=\delta d$. 
We consider the set $M_{W}$, where the number of distinct eigenvalues of $W^{+}$ is locally constant. 

Suppose $(M, g)$ is not anti-self-dual. Then $M_{W}\cap N=\varnothing$.
We note that $\Phi$ is locally a Lipschitz function. 
On $M_{W}$, we have 
\[-\Delta\Phi=\Delta\lambda_{2}=2\lambda_{2}^{2}+4\lambda_{1}\lambda_{3}-\frac{\lambda_{2}s}{2}+2(\lambda_{1}-\lambda_{2})|c|^{2}+2(\lambda_{3}-\lambda_{2})|a|^{2}.\]
Since $\lambda_{2}\leq 0$ and the scalar curvature is nonpositive, $-\frac{\lambda_{2}s}{2}\leq 0$. 
We have
\[2\lambda_{2}^{2}+4\lambda_{1}\lambda_{3}=2(\lambda_{1}+\lambda_{3})^{2}+4\lambda_{1}\lambda_{3}=2\lambda_{1}^{2}+2\lambda_{3}^{2}+8\lambda_{1}\lambda_{3}.\]
Using $\lambda_{1}+\lambda_{2}+\lambda_{3}=0$, we note that $2\lambda_{1}\lambda_{3}\leq -\lambda_{3}^{2}$ from $2\lambda_{1}\leq -\lambda_{3}$
and $\lambda_{1}^{2}\leq \lambda_{3}^{2}$ from $\lambda_{1}+\lambda_{3}\geq 0$.
It follows that 
\[2\lambda_{2}^{2}+4\lambda_{1}\lambda_{3}=2\lambda_{1}^{2}+2\lambda_{3}^{2}+8\lambda_{1}\lambda_{3}\leq 4\lambda_{3}^{2}-4\lambda_{3}^{2}=0.\]
Suppose $2\lambda_{2}^{2}+4\lambda_{1}\lambda_{3}=0$ at $p$. Then we have $2\lambda_{1}+\lambda_{3}=0$ and $\lambda_{1}+\lambda_{3}=0$ at $p$. 
Thus, $p\in N$, which is a contradiction. Thus, $2\lambda_{2}^{2}+4\lambda_{1}\lambda_{3}<0$ on $M_{W}$. 
Since $|\rho a|=|a|, |\varphi c|=|c|$, we have
\[|\nabla_{s}\Phi|^{2}=(\lambda_{2}-\lambda_{1})^{2}|c|^{2}+(\lambda_{3}-\lambda_{2})^{2}|a|^{2}+2(\lambda_{2}-\lambda_{1})(\lambda_{2}-\lambda_{3})\left<\varphi c, \rho a\right>.\] 
We consider the following
\begin{align*}
&\Psi:=\lambda_{2}\left(2(\lambda_{1}-\lambda_{2})|c|^{2}+2(\lambda_{3}-\lambda_{2})|a|^{2}\right)\\
&+k(\lambda_{2}-\lambda_{1})^{2}|c|^{2}+k(\lambda_{2}-\lambda_{3})^{2}|a|^{2}+2k(\lambda_{2}-\lambda_{1})(\lambda_{2}-\lambda_{3})\left<\varphi c, \rho a\right>.
\end{align*}
Fix a point $p\in M_{W}$. Then at $p$, 
there is an orthonormal basis $\{e_{1}, e_{2}, e_{3}, e_{4}\}$
such that 
\[\varphi=e_{1}\wedge e_{2}+e_{3}\wedge e_{4}\]
\[\rho=e_{1}\wedge e_{3}+e_{4}\wedge e_{2}.\]

Let 
\[a=a_{1}e_{1}+a_{2}e_{2}+a_{3}e_{3}+a_{4}e_{4}\]
\[c=c_{1}e_{1}+c_{2}e_{2}+c_{3}e_{3}+c_{4}e_{4}.\]
Then we have 
\[\left<\rho a , \varphi c\right>=-a_{1}c_{4}-a_{2}c_{3}+a_{3}c_{2}+a_{4}c_{1}\]

If we define again 
\[c=-c_{4}e_{1}-c_{3}e_{2}+c_{2}e_{3}+c_{1}e_{4},\]
we have 
\begin{align*}
&\Psi:=\lambda_{2}\left(2(\lambda_{1}-\lambda_{2})|c|^{2}+2(\lambda_{3}-\lambda_{2})|a|^{2}\right)\\
&+k(\lambda_{2}-\lambda_{1})^{2}|c|^{2}+k(\lambda_{2}-\lambda_{3})|a|^{2}+2k(\lambda_{2}-\lambda_{1})(\lambda_{2}-\lambda_{3})\left<c, a\right>.
\end{align*}
We show for $1-\lambda=k=\frac{2}{3}$ [11], we have $\Psi\geq 0$ on $M_{W}$. 
For this, we follow the method in [19].

Let 
\begin{align*}
&A:=[2\lambda_{2}(\lambda_{3}-\lambda_{2})+k(\lambda_{2}-\lambda_{3})^{2}]\\
&B:=[2\lambda_{2}(\lambda_{1}-\lambda_{2})+k(\lambda_{2}-\lambda_{1})^{2}]\\
&C:=k(\lambda_{2}-\lambda_{1})(\lambda_{2}-\lambda_{3}).
\end{align*}
Then we have 
\[\Psi=A|a|^{2}+B|c|^{2}+2C\left<a, c\right>.\]
Suppose $A>0$. Considering $\Psi$ as a quadratic function of $|a|$, we have
\[\Psi=A\left<a+\frac{C}{A}c\right>^{2}-\frac{C^{2}}{A}|c|^{2}+B|c|^{2}.\]
From this, the minimum of $\Psi$ is given by 
\[\left(\frac{AB-C^{2}}{A}\right)|c|^{2}.\]
Thus, $\Psi\geq 0$ if $A>0$ and $AB-C^{2}\geq 0$. 

For $k=\frac{2}{3}$, we have 
\begin{align*}
&A=\frac{2}{3}(\lambda_{3}-\lambda_{2})(2\lambda_{2}+\lambda_{3})\\
&B=-\frac{2}{3}(\lambda_{2}-\lambda_{1})(2\lambda_{2}+\lambda_{1})\\
&C=\frac{2}{3}(\lambda_{2}-\lambda_{1})(\lambda_{2}-\lambda_{3}).
\end{align*}
Then 
\[A=\frac{2}{3}(\lambda_{3}-\lambda_{2})(2\lambda_{2}+\lambda_{3})=\frac{2}{3}(\lambda_{3}-\lambda_{2})(-2\lambda_{1}-\lambda_{3})\geq 0.\]
Let $\Omega$ be the set where $2\lambda_{2}+\lambda_{3}=0$. 
Then on $M_{W}-M_{W}\cap\Omega$, $A>0$. 
We have 
\begin{align*}
&AB-C^{2}=-\frac{4}{9}(\lambda_{3}-\lambda_{2})(\lambda_{2}-\lambda_{1})
\left[(2\lambda_{2}+\lambda_{3})(2\lambda_{2}+\lambda_{1})+(\lambda_{2}-\lambda_{1})(\lambda_{3}-\lambda_{2})\right]\\
&=-\frac{4}{3}\lambda_{2}(\lambda_{3}-\lambda_{2})(\lambda_{2}-\lambda_{1})(\lambda_{1}+\lambda_{2}+\lambda_{3})=0.
\end{align*}
Thus, on $M_{W}-M_{W}\cap\Omega$, $\Psi\geq 0$.

On $M_{W}\cap\Omega$, we have $2\lambda_{2}+\lambda_{3}=0$. Thus, $\lambda_{1}=\lambda_{2}$. 
It follows that $A=B=C=0$, and therefore $\Psi=0$.
Thus, on $M_{W}$, $\Psi\geq0$.
On $M_{W}$, $\Phi=-\lambda_{2}>0$ if $W^{+}\not\equiv 0$ by Lemma 6, and therefore, we have 
\[\Delta\Phi^{1/3}=\frac{1}{3}\Phi^{-5/3}(\Phi\Delta\Phi+\frac{2}{3}\Sigma|\nabla_{s}\Phi|^{2})>0.\]
Thus, $\Phi^{1/3}$ is a strictly superharmonic function on $M_{W}$. 

We note that $\lambda_{2}$ is continuous on $M$ and differentiable on $M_{W}$. 
Since $M_{W}$ is an open dense subset of $M$, each piece of $M-M_{W}$ is a hypersurface or codimension $\geq 2$.
$\lambda_{2}$ may not be differentiable on $\partial{M_{W}}$, 
but $\frac{\partial\lambda_{2}}{\partial{\nu}}$ exists on $\partial{M_{W}}-N$, where $\nu$ is the outward unit normal vector. 

Suppose $\lambda_{2}=0$ at $p\in\partial M_{W}-N$. 
Then the number of distinct eigenvalues of $W^{+}$ is 3, which implies $p\in M_{W}$, which is a contradiction. 
Thus, $\lambda_{2}<0$ on $\partial M_{W}-N$. 
Since $\lambda_{3}$ is a smooth function on $M-N$ and each piece of  $\partial{M_{W}}-N$ is counted twice with opposite directions, it follows that 
\[\int_{\partial{M_{W}}-N}(-\lambda_{2})^{-1}\frac{\partial\lambda_{3}}{\partial{\nu}}da=0.\]
On $\partial{M_{W}}-N$, $2\lambda_{1}+\lambda_{3}=0$, while $2\lambda_{1}+\lambda_{3}\leq 0$ on $M$. 
Thus, for an outward normal vector $\nu$ on $\partial{M_{W}}-N$, we have $\frac{\partial(2\lambda_{1}+\lambda_{3})}{\partial{\nu}}\geq 0$. 
Using $-\lambda_{2}=\frac{1}{2}\left(\lambda_{3}+(2\lambda_{1}+\lambda_{3})\right)$, we have 
\[\int_{\partial{M_{W}}-N}\Phi^{-1}\frac{\partial\Phi}{\partial{\nu}}da=\frac{1}{2}\int_{\partial{M_{W}}-N}(-\lambda_{2})^{-1}\frac{\partial\lambda_{3}}{\partial{\nu}}da
+\frac{1}{2}\int_{\partial{M_{W}}-N}(-\lambda_{2})^{-1}\frac{\partial (2\lambda_{1}+\lambda_{3})}{\partial{\nu}}da\]
\[=\frac{1}{2}\int_{\partial{M_{W}}-N}(-\lambda_{2})^{-1}\frac{\partial (2\lambda_{1}+\lambda_{3})}{\partial{\nu}}da\geq 0.\]

Let $Y_{\varepsilon}:=\{x\in M| dist_{g}(x, \partial M_{W})\geq\varepsilon>0\}$. 
We have 
\[\int_{Y_{\varepsilon}}\Phi^{-1/3}\Delta\Phi^{1/3}=\int_{Y_{\varepsilon}}\left<d\Phi^{1/3}, d\Phi^{-1/3}\right>d\mu
-\int_{\partial Y_{\varepsilon}}\Phi^{-1/3}\frac{\partial\Phi^{1/3}}{\partial{\nu}}da>0.\]
It follows that 
\[-\int_{\partial Y_{\varepsilon}}\Phi^{-1}\frac{\partial\Phi}{\partial{\nu}}da>\frac{1}{3}\int_{Y_{\varepsilon}}\Phi^{-2}|d\Phi|^{2}d\mu.\]
Fix $\varepsilon>0$. Then 
\[\int_{\partial Y_{\varepsilon}}\Phi^{-2}\left(\frac{\partial\Phi}{\partial{\nu}}\right)^{2}da\leq C_{\varepsilon},\]
where $C_{\varepsilon}$ is a constant which depends on $\varepsilon$.
By the Cauchy-Schwarz inequality, we have 
\[\left(\int_{\partial Y_{\varepsilon}}da\right)^{1/2}\left(\int_{\partial Y_{\varepsilon}}\Phi^{-2}\left(\frac{\partial\Phi}{\partial{\nu}}\right)^{2}da\right)^{1/2}\geq
\left|\int_{\partial Y_{\varepsilon}}\Phi^{-1}\frac{\partial\Phi}{\partial{\nu}}da\right|\]
\[\geq\frac{1}{3}\int_{Y_{\varepsilon}}\Phi^{-2}|d\Phi|^{2}d\mu\geq\frac{1}{3}\int_{\partial Y_{\varepsilon}}\Phi^{-2}|d\Phi|^{2}d\mu
\geq\frac{1}{3}\int_{\partial Y_{\varepsilon}}\Phi^{-2}\left(\frac{\partial\Phi}{\partial{\nu}}\right)^{2}da\]
We note that 
\[\left(\int_{\partial Y_{\varepsilon}}da\right)\leq C,\]
where $C>0$ is a constant which is independent of $\varepsilon$. 
Then we get, for every $\varepsilon$, 
\[\int_{\partial Y_{\varepsilon}}\Phi^{-2}\left(\frac{\partial\Phi}{\partial{\nu}}\right)^{2}da\leq 9C.\]
Thus, for every $\varepsilon$, we have
\[\left|\int_{\partial Y_{\varepsilon}}\Phi^{-1}\frac{\partial\Phi}{\partial{\nu}}da\right|\leq 3C.\]
Let $X_{\varepsilon}$ be the $\varepsilon$-neighborhood of $\partial M_{W}\cap N$. 
We have 
\[9C\geq\int_{\partial Y_{\varepsilon}}\Phi^{-2}\left(\frac{\partial\Phi}{\partial{\nu}}\right)^{2}da\geq
\int_{\partial X_{\varepsilon}}\Phi^{-2}\left(\frac{\partial\Phi}{\partial{\nu}}\right)^{2}da\]
By the Cauchy-Schwarz inequality, we have 
\[\left(\int_{\partial X_{\varepsilon}}da\right)^{1/2}\left(\int_{\partial X_{\varepsilon}}\Phi^{-2}\left(\frac{\partial\Phi}{\partial{\nu}}\right)^{2}da\right)^{1/2}\geq
\left|\int_{\partial X_{\varepsilon}}\Phi^{-1}\frac{\partial\Phi}{\partial{\nu}}da\right|\]
By the Bolzano-Weierstrass theorem, a bounded sequence has a convergent subsequence.
Thus, $\int_{\partial Y_{\varepsilon}}\Phi^{-1}\frac{\partial\Phi}{\partial{\nu}}da$ and $\int_{\partial X_{\varepsilon}}\Phi^{-1}\frac{\partial\Phi}{\partial{\nu}}da$
have convergent subsequences as $\varepsilon\to 0$.  
We have 
\[\lim_{\varepsilon\to 0}\int_{\partial Y_{\varepsilon}}\Phi^{-1}\frac{\partial\Phi}{\partial{\nu}}da
=\int_{\partial M_{W}-N}\Phi^{-1}\frac{\partial\Phi}{\partial{\nu}}da
+\lim_{\varepsilon\to 0}\int_{\partial X_{\varepsilon}}\Phi^{-1}\frac{\partial\Phi}{\partial{\nu}}da.\]
Since $\int_{\partial M_{W}-N}\Phi^{-1}\frac{\partial\Phi}{\partial{\nu}}da\geq 0$, 
we have 
\[-\lim_{\varepsilon\to 0}\int_{\partial X_{\varepsilon}}\Phi^{-1}\frac{\partial\Phi}{\partial{\nu}}da
\geq -\lim_{\varepsilon\to 0}\int_{\partial Y_{\varepsilon}}\Phi^{-1}\frac{\partial\Phi}{\partial{\nu}}da\geq\frac{1}{3}\int_{M_{W}}\Phi^{-2}|d\Phi|^{2}d\mu.\]
Then we have
\[\lim_{\varepsilon\to 0}\left(\int_{\partial X_{\varepsilon}}da\right)^{1/2}\lim_{\varepsilon\to 0}\left(\int_{\partial X_{\varepsilon}}\Phi^{-2}\left(\frac{\partial\Phi}{\partial{\nu}}\right)^{2}da\right)^{1/2}\]
\[\geq \lim_{\varepsilon\to 0}\left|\int_{\partial X_{\varepsilon}}\Phi^{-1}\frac{\partial\Phi}{\partial{\nu}}da\right|\geq\frac{1}{3}\int_{M_{W}}\Phi^{-2}|d\Phi|^{2}d\mu\]
Since $\int_{\partial X_{\varepsilon}}\Phi^{-2}\left(\frac{\partial\Phi}{\partial{\nu}}\right)^{2}da$ is bounded 
and $\lim_{\varepsilon\to 0}\left(\int_{\partial X_{\varepsilon}}da\right)^{1/2}=0$, we get $\Phi^{-2}|d\Phi|^{2}=0$. 
Since $\Phi>0$ on $X$, $d\Phi=0$ on $M_{W}$, which is a contradiction since $\Phi^{1/3}$ is strictly superharmonic on $M_{W}$. 
Thus, $(M, g)$ is anti-self-dual.
\end{proof}

\begin{Theorem}
Let $(M, g, \omega)$ be a compact, oriented, four-manifold with nonnegative scalar curvature. 
Suppose $\delta W^{+}=0$ and $\frac{s}{12}\leq\lambda_{1}+\lambda_{3}$, 
where $\lambda_{i}$ are the eigenvalues of $W^{+}$ such that $\lambda_{1}\leq\lambda_{2}\leq\lambda_{3}$. 
Then 
\begin{itemize}
\item $(M, g)$ is anti-self-dual with zero scalar curvature; or
\item $b_{+}(M)=1$ and $(M, g)$ is K\"ahler with positive constant scalar curvature; or
\item $b_{+}(M)=0$ and the double cover of $(M, g)$ is K\"ahler metric with positive constant scalar curvature. 
\end{itemize}
\end{Theorem}
\begin{proof}
Suppose $(M, g)$ is not anti-self-dual. Then $M_{W}\cap N=\varnothing$, where $N$ is the zero set of $W^{+}$. 
Since $s\geq 0$, $\lambda_{1}+\lambda_{3}\geq 0$ and therefore, $\lambda_{2}\leq 0$.
By Lemma 6, $\lambda_{2}<0$ on $M-N$. 
We follow the proof of Theorem 15. 
Since $\lambda_{1}\leq\lambda_{2}\leq 0$ and $\lambda_{2}\leq-\frac{s}{12}$, 
we have 
\begin{align*}
&2\lambda_{2}^{2}+4\lambda_{1}\lambda_{3}-\frac{\lambda_{2}s}{2}=2\lambda_{2}^{2}-4\lambda_{1}^{2}-4\lambda_{1}\lambda_{2}-\frac{\lambda_{2}s}{2}\\
&\leq 2\lambda_{2}^{2}-4\lambda_{2}^{2}-4\lambda_{1}\lambda_{2}-\frac{\lambda_{2}s}{2}
=-\lambda_{2}\left(2\lambda_{2}+4\lambda_{1}+\frac{s}{2}\right)
\leq-\lambda_{2}\left(6\lambda_{2}+\frac{s}{2}\right)\leq 0.\\
\end{align*}
From this, we get $\Delta\Phi^{1/3}\geq 0$. 
Then by the same argument with Theorem 15, we get $\Delta\Phi^{1/3}=0$ on $M_{W}$.
Since $\lambda_{2}<0$ on $M_{W}$ and $\Phi\Delta\Phi-\Psi=0$,
it follows that $\lambda_{1}=\lambda_{2}=-\frac{s}{12}$ on $M_{W}$.
Since $\lambda_{2}+\frac{s}{12}$ is continuous on $M$ and $\lambda_{2}+\frac{s}{12}=0$ on $M_{W}$, 
it follows that $\lambda_{2}+\frac{s}{12}=0$ on $M$ and similarly, $\lambda_{1}+\frac{s}{12}=0$ on $M$.
Then $\lambda_{3}=\frac{s}{6}$. 
Since 
\[\Delta|W^{+}|^{2}=-2|\nabla W^{+}|^{2}+36det W^{+}-s|W^{+}|^{2},\]
 we get $s$ is constant.
The result follows from [7], [15].
\end{proof}

\begin{Theorem}
(Guan [11]) Let $(M, g, \omega)$ be a compact K\"ahler-Einstein surface with nonpositive scalar curvature. 
Suppose 
\[H_{av}-H_{min}\leq \frac{1}{2}(H_{max}-H_{min}),\]
where $H_{max}(H_{min})$ denotes the maximum(minimum) of the holomorphic sectional curvature at each point respectively 
and $H_{av}$ the average
of the holomorphic sectional curvature at each point. 
Then $M$ is a compact quotient of complex 2-dimensional unit ball or plane. 
\end{Theorem}
\begin{proof}
By Lemma 3,  
\[H_{av}-H_{min}\leq \frac{1}{2}(H_{max}-H_{min})\]
is equivalent to $\lambda_{1}+\lambda_{3}\geq 0$, where $\lambda_{i}$ are eigenvalues of $W^{-}$ such that 
$\lambda_{1}\leq \lambda_{2}\leq \lambda_{3}$ and therefore $det W^{-}\geq 0$. 
Since an Einstein metric implies $\delta W^{\pm}=0$, we get $(M, g)$ is self-dual from Theorem 15. 
Thus, $(M, g, \omega)$ is a self-dual K\"ahler-Einstein surface with nonpositive scalar curvature.
By Lemma 3, $(M, g, \omega)$ has constant holomorphic sectional curvature. 
\end{proof}

\vspace{20pt}

$\mathbf{Acknowledgements}$: The author is very thankful to the referee for pointing out
errors in the first version of this paper.
The author would like to thank Prof. Claude LeBrun for precious comments.


\newpage

\begin{thebibliography}{9}


\bibitem{AD}
V. Apostolov and J. Davidov, 
\emph{Compact Hermitian surfaces and isotropic curvature}, 
Illinois Jour. of Math. Vol 44, Number 2, summer 2000, 438-451.

\bibitem{B}
C. B\"ar, 
\emph{On nodal sets for Dirac and Laplace operators}, 
Comm. Math. Phys. Vol. 188, 709-721 (1997). 

\bibitem{BPV}
W. Barth, C. Peters, and A. Van de Ven, 
\emph{Compact complex surfaces}, 
Vol. 4 of Ergebnisse der Mathematik und ihrer Grenzgebiete  (3), Springer-Verlag, Berlin (1984). 

\bibitem{B}
M. Berger, 
\emph{Sur les variétés d'Einstein compactes}, 
Comptes Rendus de la $\RN{3}$e Réunion du Groupement des Mathématiciens d'Espression Latine, (Namur, 1965) pp.35-55. 


\bibitem{B}
A. Besse, Einstein manifolds, Springer, Berlin (1987). 

\bibitem{CR}
E. Costa and E. Ribeiro Jr., 
\emph{Four-dimensional compact manifolds with nonnegative biorthogonal curvature}, 
Michigan Math. J. 63 (2014), 747-761. 



\bibitem{D}
A. Derdziński,
\emph{Self-dual K\"ahler manifolds and Einstein manifolds of dimension four}, 
Compos. Math. 49, 405-433 (1983). 

\bibitem{F}
R. Friedman, 
\emph{Algebraic surfaces and holomorphic vector bundles}, 
Springer-Verlag, New York, 1998.


\bibitem{FK}
T. Friedrich and H. Kurke
\emph{Compact four-dimensional self-dual Einstein manifolds with positive scalar curvature}, 
Math. Nachr. 106 (1982) 271-299.

\bibitem{G}
A. Gray, 
\emph{Compact K\"ahler manifolds with nonnegative sectional curvature}, 
Invent. math. 41, 33-43 (1977). 




\bibitem{G}
D. Guan, 
\emph{On bisectional nonpositively curved compact K\"ahler-Einstein surfaces}, 
Pacific Journal of Math. Vol. 288, No.2, 2017, 343-353. 

\bibitem{GL}
M. Gursky and C. LeBrun, 
\emph{On Einstein manifolds of positive sectional curvature}, 
Ann. Global Anal. Geom. 17 (1999): 315-328. 


\bibitem{HM}
S. J. Hall and T. Murphy, 
\emph{Rigidity results for Hermitian-Einstein manifolds}, 
Math. Proc. R. Ir. Acad. Vol. 116 A , No.1 (2016), pp. 35-44.


\bibitem{H}
N. J. Hitchin, 
\emph{K\"ahlerian Twistor spaces}, 
Proc. Lond. Math. Soc. 43 (1981) 133-150. 

\bibitem{L}
C. LeBrun, 
\emph{Einstein manifolds, self-dual Weyl curvature, and conformally K\"ahler geometry}, 
Math. Res. Lett. 28 (2021), 127-144. 

\bibitem{M}
Y. Matsushima, 
\emph{Sur la structure du groupe d'homéomorphismes d'une certaine varieté K\"ahlérienne}, 
Nagoya Math. Journ. 11 (1957) 145- 150. 

\bibitem{P}
A. Polombo, 
\emph{De nouvelles formules de Weitzenb\"ock pour des endomorphismes harmoniques. Applications géométriques}, 
Ann. Scient. $\acute{E}$c. Norm. Sup., $4^{e}$ s$\acute{e}$rie, t. 25, 1992, 393-428. 

\bibitem{SY}
Y. T. Siu and P. Yang, 
\emph{Compact K\"ahler-Einstein surfaces of nonpositive bisectional curvature}, 
Invent. math. 64, 471-487 (1981). 

\bibitem{W}
P. Wu, 
\emph{Einstein four-manifolds with self-dual Weyl curvature of nonnegative determinant}, 
Int. Math. Res. Not. Vol. 2021, No. 2, pp. 1043-1054.


\bibitem{Y}
S. -T. Yau, 
\emph{On the curvature of compact Hermitian manifolds}, 
Invent. math. 25 (1974), 213-239. 

\bibitem{Z}
F. Zheng, 
\emph{First Pontrjagin form, rigidity and strong rigidity of nonpositively curved K\"ahler surface of general type}, 
Math. Z. 220, 159-169 (1995).




\end{thebibliography}
\end{document}